\documentclass[leqno,12pt]{article} 
\usepackage{amsmath, amssymb}
\usepackage{amsthm} 
\theoremstyle{plain} 
\newtheorem{theorem}{\indent\sc Theorem}[section]
\newtheorem{lemma}[theorem]{\indent\sc Lemma}

\newtheorem{proposition}[theorem]{\indent\sc Proposition}

\theoremstyle{definition} 
\newtheorem{definition}[theorem]{\indent\sc Definition}
\newtheorem{remark}[theorem]{\indent\sc Remark}

\title{Generalized metallic structures}
\author{Adara M. Blaga and Antonella Nannicini}

\date{}
\begin{document}

\maketitle

\markboth{{\small\it {\hspace{4cm} Generalized metallic structures}}}{\small\it{Generalized metallic structures
\hspace{4cm}}}


\begin{abstract}
 We study the properties of a generalized metallic, a generalized product and a generalized complex structure induced on the generalized tangent bundle of $M$ by a metallic Riemannian structure $(J,g)$ on $M$, providing conditions for their integrability with respect to a suitable connection. Moreover, by using methods of generalized geometry, we lift $(J,g)$ to metallic Riemannian structures on the tangent and cotangent bundles of $M$, underlying the relations between them.
\end{abstract}

\bigskip

\section{Preliminaries}

On a smooth manifold $M$, besides the almost complex, almost tangent, almost product structures etc., some other polynomial structures can be considered as $C^{\infty}$-tensor fields $J$ of $(1,1)$-type which are roots of the algebraic equation
$$Q(J):=J^n+a_nJ^{n-1}+\dots+a_2J+a_1I=0,$$
where $I$ is the identity operator on the Lie algebra of vector fields on $M$. In particular, if $Q(J):=J^2-pJ-qI$, with $p$ and $q$ positive integers, its solution $J$ will be called \textit{metallic structure} \cite{a}. The name is motivated by the fact that the \textit{$(p,q)$-metallic number} introduced by Vera W. de Spinadel \cite{s} is precisely the positive root of the quadratic equation $x^2-px-q=0$, namely $\sigma_{p,q}:=\frac{\displaystyle p+\sqrt{p^2+4q}}{\displaystyle 2}$.
For example: if $p=q=1$ we get the \textit{golden number} $\sigma=\frac{\displaystyle 1+\sqrt{5}}{\displaystyle 2}$; if $p=2$ and $q=1$ we get the \textit{silver number} $\sigma_{2,1}=1+\sqrt{2}$; if $p=3$ and $q=1$ we get the \textit{bronze number} $\sigma_{3,1}=\frac{\displaystyle 3+\sqrt{13}}{\displaystyle 2}$; if $p=1$ and $q=2$ we get the \textit{copper number} $\sigma_{1,2}=2$; if $p=1$ and $q=3$ we get the \textit{nickel number} $\sigma_{1,3}=\frac{\displaystyle 1+\sqrt{13}}{\displaystyle 2}$ and so on.\\

We shall briefly recall the basic notions of metallic (Riemannian) geometry.

\begin{definition}\cite{c} \label{d1}
\textit{A metallic structure} $J$ on $M$ is an endomorphism $J:TM\rightarrow TM$ satisfying
\begin{equation}\label{0.1}
J^2=pJ+qI,
\end{equation}
for some $p$, $q\in\mathbb{N}^*$. The pair $\left(M, J\right) $ is called a {\it metallic manifold}. Moreover, if a Riemannian metric $g$ on $M$ is compatible with $J$, that is $g(JX, Y)=g(X, JY)$, for any $X$, $Y\in C^{\infty}(TM)$, we call the pair $(J,g)$ a {\it metallic Riemannian structure} and $(M, J,g)$ a {\it metallic Riemannian manifold}.
\end{definition}

The concept of integrability for a metallic structure is defined in the classical manner.

\begin{definition}\label{d3}
A metallic structure $J$ is called {\it integrable} if its Nijenhuis tensor field:
$$N_{J}(X, Y):=[JX, JY]-J[JX, Y]-J[X, JY] +J^{2}[X, Y]$$ vanishes for all $X,Y\in C^{\infty}(TM)$.
\end{definition}


It is known \cite{c} that an almost product structure $F$ on $M$ induces two metallic structures:
$$
J^{\pm}=\pm \frac{2\sigma _{p, q}-p}{2}F+\frac{p}{2}I
$$
and, conversely, every metallic structure $J$ on $M$ induces two almost product structures:
$$
F^{\pm }=\pm (\frac{2}{2\sigma _{p, q}-p}J-\frac{p}{2\sigma _{p, q}-p}I),
$$
where $\sigma _{p,q}=\frac{\displaystyle p+\sqrt{p^{2}+4q}}{\displaystyle 2}$ is the metallic number, for $p,q\in \mathbb{N}^*$.

In particular, if the almost product structure $F$ is compatible with a Riemannian metric $g$, then $(J^{+},g)$ and $(J^{-},g)$ are metallic Riemannian structures.\\

The analogue concept of locally product manifold is considered in the context of metallic geometry.

\begin{definition} \cite{b}
A metallic Riemannian manifold $(M,J,g)$ is called \textit{locally metallic} if $J$ is parallel with respect to the Levi-Civita connection $\nabla$ of $g$, that is $\nabla J=0$.
\end{definition}

In the followings, we shall extend the definition of a metallic structure for any $p$ and $q$ real numbers. In this way, we also include some other well-known structures; for instance, if $(p,q) \in \{ (0,-1), (0,0), (0,1), (1,0) \}$, the solution of (\ref{0.1}) would yield an almost complex, an almost tangent, an almost product and a $J(2,1)$-structure, respectively.

\section{Generalized structures induced by metallic structures}

Let $TM\oplus T^*M$ be the generalized tangent bundle of a smooth manifold $M$.

\begin{definition}\label{...}
\textit{A generalized metallic structure} $\hat{J}$ on $M$ is an endomorphism $\hat{J}:TM\oplus T^*M\rightarrow TM\oplus T^*M$ satisfying
$$\hat{J}^2=p\hat{J}+qI,$$
for some real numbers $p$ and $q$.
\end{definition}

For a linear connection $\nabla$ on $M$, we consider the bracket $[\cdot,\cdot]_{\nabla}$ on $C^{\infty}(TM\oplus T^*M)$ \cite{na}:
$$[X+\alpha,Y+\beta]_{\nabla}:=[X,Y]+\nabla_X\beta-\nabla_Y\alpha,$$
for all $X,Y\in C^{\infty}(TM)$ and $\alpha,\beta\in C^{\infty}(T^*M)$.

\begin{definition}\label{...}
A generalized metallic structure $\hat{J}$ is called {\it $\nabla$-integrable} if its Nijenhuis tensor field $N_{\hat{J}}^{\nabla}$ with respect to $\nabla$:
$$N_{\hat{J}}^{\nabla}(\sigma, \tau):=[\hat{J}\sigma,\hat{J}\tau]_{\nabla}-\hat{J}[\hat{J}\sigma, \tau]_{\nabla}-\hat{J}[\sigma, \hat{J}\tau]_{\nabla} +\hat{J}^{2}[\sigma, \tau]_{\nabla}$$ vanishes for all $\sigma,\tau \in C^{\infty}(TM\oplus T^*M)$.
\end{definition}

\subsection{Generalized metallic structure induced by $(J,g)$}

Let $(J,g)$ be a metallic Riemannian structure on $M$ such that $J^2=pJ+qI$, $p,q\in \mathbb{R}$. If we denote by $\sharp_g:T^*M\rightarrow TM$ the inverse of the isomorphism $\flat_g:TM\rightarrow T^*M$, $\flat_g(X):=i_Xg$, from the $g$-symmetry of $J$ we have $\sharp_g\circ J^*=J\circ \sharp_g$ and $\flat_g\circ J=J^*\circ \flat_g$, where $(J^*\alpha)(X):=\alpha(JX)$. Also notice that $J^*$ is a metallic structure, too, namely, $(J^*)^2=pJ^*+qI$, and we easily get that $\sharp_g\circ (J^*)^k=J^k\circ \sharp_g$ and $\flat_g\circ J^k=(J^*)^k\circ \flat_g$, for any $k\in \mathbb{N}$.

On $TM\oplus T^*M$ we consider the Riemannian metric:
\begin{equation}\label{e}
\hat{g}(X+\alpha,Y+\beta):=g(X,Y)+g(\sharp_g\alpha,\sharp_g\beta),
\end{equation}
for any $X,Y\in C^{\infty}(TM)$ and $\alpha,\beta\in C^{\infty}(T^*M)$.

\begin{definition}
A pair $(\hat{J},\hat{g})$ of a generalized metallic structure $\hat{J}$ and a Riemannian metric $\hat{g}$ such that $\hat{J}$ is $\hat{g}$-symmetric is called $\textit{generalized metallic Riemannian structure}$.
\end{definition}

Remark that the generalized metallic structure
$\hat{J}_m:=\begin{pmatrix}
               J & 0 \\
               0 & J^* \\
         \end{pmatrix}$
induced by the metallic Riemannian structure $(J,g)$ is $\hat{g}$-symmetric, hence, $(\hat{J}_m,\hat{g})$ is a generalized metallic Riemannian structure.

\begin{proposition}
The generalized metallic structure $\hat{J}_m$ induced by the metallic Riemannian structure $(J,g)$ on $M$ is $\nabla$-integrable if and only if $J$ is integrable and $(\nabla_{JX}J)=(\nabla_{X}J)J$, for any $X\in C^{\infty}(TM)$.
\end{proposition}

\begin{proof}  We have:\\

$N_{\hat{J}_m}^{\nabla}(X,Y)=[JX,JY]-J[JX, Y]-J[X, JY] +J^{2}[X,Y]=N_J(X,Y)$\\

$N_{\hat{J}_m}^{\nabla}(X, \beta)=[JX,J^*\beta]_{\nabla}-J^*[JX, \beta]_{\nabla}-J^*[X,J^*\beta]_{\nabla} +(J^*)^{2}[X, \beta]_{\nabla}$\\

$={\nabla}_{JX} J^{*} \beta-J^* \nabla_{JX} \beta-J^*\nabla_XJ^* \beta+(J^*)^2\nabla_X \beta$\\

$=((\nabla_{JX}J^*)-J^*(\nabla_X J^*))(\beta)$\\

$=\beta((\nabla_{JX}J)-(\nabla_X J)J)$\\

$N_{\hat{J}_m}^{\nabla}(\alpha, \beta)=0$,\\

\noindent for all $X,Y\in C^{\infty}(TM)$ and $\alpha,\beta\in C^{\infty}(T^*M)$. Then the proof is complete.
\end{proof}

Remark that if $\nabla$ is a \textit{$J$-connection}, that is $\nabla J=0$, then $\hat{J}_m$ is $\nabla$-integrable if and only if $J$ is integrable. Moreover, if  $T^{\nabla}$ is the torsion of ${\nabla}$, $T^{\nabla}(X,Y):={\nabla}_X Y-{\nabla}_Y X -[X,Y] $, then a direct computation gives:
$$N_{J}(X, Y)=(\nabla_{JX}J)Y-(\nabla_{JY}J)X+J(\nabla_{Y}J)X-J(\nabla_{X}J)Y+\Phi (T^{\nabla})(X,Y),$$
where:
$$\Phi (T^{\nabla})(X,Y):=-T^{\nabla}(JX,JY)+JT^{\nabla}(JX,Y)+JT^{\nabla}(X,JY)-J^2T^{\nabla}(X,Y).$$
In particular, if $\nabla$ is a torsion free $J$-connection, then $\hat{J}_m$ is $\nabla$-integrable.\\

Let $\nabla^g$ be the Levi-Civita connection of $g$ and define a linear connection $D$ on $M$ by $D:=\nabla^g+F$, where $F$ is a $(1,2)$-type tensor field such that
$$\left\{
    \begin{array}{ll}
      DJ=0 \\
      Dg=0.
    \end{array}
  \right.
$$
This is equivalent to
$$\left\{
    \begin{array}{ll}
      (\nabla^g_XJ)Y= J(F(X,Y))-F(X,JY)\\
      g(F(X,Y),Z)+g(Y,F(X,Z))=0
    \end{array}
  \right.,
$$
for any $X,Y,Z\in C^{\infty}(TM)$.\\

Consider the bracket $[\cdot,\cdot]_{D}$ on $C^{\infty}(TM\oplus T^*M)$ \cite{na}:
$$[X+\alpha,Y+\beta]_{D}:=[X,Y]+D_X\beta-D_Y\alpha,$$
for any $X,Y\in C^{\infty}(TM)$ and $\alpha,\beta\in C^{\infty}(T^*M)$.

Define the connection $\hat{D}$ on $TM\oplus T^*M$ by \cite{n}:
$$\hat{D}_X(Y+\beta):=D_XY+D_X\beta,$$
for any $X,Y\in C^{\infty}(TM)$ and $\beta\in C^{\infty}(T^*M)$. It follows that:
$$\hat{D}_X(Y+\beta)=\nabla^g_XY+F(X,Y)+\nabla_X\beta-\beta\circ F(X,\cdot).$$

Let $n$ be the dimension of $M$ and assume that $q\neq 0$. Denote by $\{x^1,...,x^n\}$ the local coordinates on $M$ and let $\{X_1,...,X_n\}$ be the corresponding local frame for $TM$. Following \cite{k} we define:
$$F(X_i,X_j):=\omega(X_j)X_i-\omega(X_l)g^{lk}g_{ij}X_k+ {1\over q}\omega(JX_j)JX_i-{1\over q}\omega(JX_l)g^{lk}J^s_j g_{is}X_k,$$
where $\omega$ is a $1$-form on $M$ and we use Einstein's convention of summation.

We immediately have that $g(F(X_i,X_j),X_r)+g(X_j,F(X_i,X_r))=0$, for all $i,j,r$, therefore, $Dg=0$, for any $1$-form $\omega$. Moreover, the torsion of $D$ is given by:
$$T^D(X,Y)=\omega(Y)X-\omega (X)Y+ {1\over q}\omega (JY)JX-{1\over q}\omega (JX)JY,$$
for any $X,Y \in C^{\infty}(TM).$

\begin{lemma}
$T^D$ satisfies the following properties:
$$T^D(JX,Y)=JT^D(X,Y)=T^D(X,JY)$$
$$\Phi(T^D)(X,Y)=0,$$
for any $X,Y\in C^{\infty}(TM)$.
\end{lemma}

\begin{proof}
From a direct computation we get:
$$JT^D(X,Y)=\omega(Y)JX-\omega (X)JY+ {p\over q}\omega (JY)JX-{p\over q}\omega (JX)JY+\omega(JY)X-\omega(JX)Y$$
which is equal to $T^D(JX,Y)$ and $T^D(X,JY)$.

Consequently, we have $\Phi (T^{\nabla})(X,Y)=0$.
\end{proof}

Recently C. Karaman \cite{k} constructed metallic semi-symmetric metric $J$-connections $D$ on locally decomposable metallic Riemannian manifolds $(M,J,g)$. These connections satisfy:
$$DJ=0, \ Dg=0, \ T^D(X,Y)=\omega(Y)X-\omega (X)Y+ {1\over q}\omega (JY)JX-{1\over q}\omega (JX)JY,$$
for any $X,Y \in C^{\infty}(TM).$ In particular, we can state the following:

\begin{proposition} Let $(M,J,g)$ be a locally decomposable metallic Riemannian manifold and let $D$ be a metallic semi-symmetric metric $J$-connection. Then ${\hat J_m}$ is $D$-integrable.
\end{proposition}

\begin{proposition}
Let $(\hat{J}_m:=\begin{pmatrix}
               J & 0 \\
               0 & J^* \\
         \end{pmatrix}, \hat{g})$ be the generalized metallic Riemannian structure induced by the metallic Riemannian structure $(J,g)$ on $M$ with $\hat{g}$ the Riemannian metric defined by (\ref{e}). Then:
         \begin{enumerate}
           \item $\hat{D}\hat{J}_m=0$ if and only if $DJ=0$;
           \item $\hat{D}\hat{g}=0$ if and only if $Dg$.
         \end{enumerate}
\end{proposition}

\begin{proof}
Remark that $\hat{D}\hat{J}_m=0$ is equivalent to $(D_XJ)Y+\beta\circ D_XJ=0$, for any $X,Y\in C^{\infty}(TM)$ and $\beta\in C^{\infty}(T^*M)$ and
$\hat{D}\hat{g}=0$ is equivalent to $(D_Xg)(Y,Z)-(D_Xg)(\sharp_g \beta,\sharp_g \gamma)=0$, for any $X,Y,Z\in C^{\infty}(TM)$ and $\beta, \gamma\in C^{\infty}(T^*M)$.
\end{proof}

\begin{definition} A smooth map $f$ between two metallic manifolds $(M_1,J_1)$ and $(M_2,J_2)$ is called \emph{metallic} if $f_*\circ J_1=J_2\circ f_*$.
\end{definition}

\begin{remark}
A metallic diffeomorphism $f$ between two metallic manifolds $(M_1,J_1)$ and $(M_2,J_2)$ naturally induces an isomorphism $\hat{f}$ between their generalized tangent bundles defined by:
$$\hat{f}:TM_1\oplus T^*M_1\rightarrow TM_2\oplus T^*M_2, \ \ \hat{f}(X+\alpha):=f_*X+((f_*)^*)^{-1}\alpha,$$
where $f_*:TM_1\rightarrow TM_2$ is the tangent map of $f$ and $(f_*)^*:T^*M_2\rightarrow T^*M_1$ is the dual map of $f_*$, that is $((f_*)^*\alpha)(X):=\alpha(f_*X)$, for all $\alpha\in C^{\infty}(T^*M_2)$ and $X\in C^{\infty}(TM_1)$, which preserves the generalized metallic structures $\hat{J}_{i,m}:=\begin{pmatrix}
               J_i & 0 \\
               0 & J_i^* \\
         \end{pmatrix}$, $i=1,2$. Indeed, from $f_*\circ J_1=J_2\circ f_*$ follows $(f_*)^*\circ J_2^*=J_1^*\circ (f_*)^*$, hence $\hat{f}\circ \hat{J}_{1,m}= \hat{J}_{2,m}\circ \hat{f}$.

In particular, if $f:M\rightarrow M$ is a diffeomorphism which preserves the metallic structure $J$, then $\hat{f}$ can be defined by
$$\hat{f}(X+\alpha):=f_*X+(f_*)^*\alpha$$
which coincides with the generalized metallic structure $\hat{J}_m$ when $J=f_*$. In this case, $J$ is invertible and $J^{-1}=\frac{\displaystyle 1}{\displaystyle q}J-\frac{\displaystyle p}{\displaystyle q}I$, for $q\neq 0$.
\end{remark}

\subsection{Generalized product structure induced by $(J,g)$}

Let $(J,g)$ be a metallic Riemannian structure on $M$ such that $J^2=pJ+qI$, $p,q\in \mathbb{R}$. Then $\hat{J}_p:=\begin{pmatrix}
               J & (I-J^2)\sharp_g \\
               \flat_g & -J^* \\
         \end{pmatrix}$ is a generalized product structure on $M$, that is $\hat{J}_p^2=I$.\\

A direct computation gives the following.

\begin{proposition}
The generalized product structure $\hat{J}_p$ induced by the metallic Riemannian structure $(J,g)$ on $M$ is $\nabla$-integrable if and only if the following conditions are satisfied:\\

$ N_J-(I-J^2)\sharp_g(d^{\nabla}g)=0$\\

$(\nabla_{JX}g)Y- (\nabla_{JY}g)X+J^*( (\nabla_{X}g)Y- (\nabla_{Y}g)X)+g( (\nabla_{Y}J)X- (\nabla_{X}J)Y)+$\\

$+g(T^{\nabla}(X,JY)+T^{\nabla}(JX,Y))=0$\\

$(d^{\nabla}g)((I-J^2)Y,X)-(\nabla_{X}J^*)g(JY)+(\nabla_{JX}J^*)g(Y)=0$\\

$(\nabla_{(I-J^2)X}J^*)g(Y)- (\nabla_{(I-J^2)Y}J^*)g(X)=0$\\

$(\nabla_{(I-J^2)X}J^2)Y- (\nabla_{(I-J^2)Y}J^2)X+T^{\nabla}((I-J^2)X,(I-J^2)Y)+$\\

$-(I-J^2)\sharp_g((\nabla_{(I-J^2)X}g)Y- (\nabla_{(I-J^2)Y}g)X)=0$\\

$-(\nabla_{JX}J^2)Y-(\nabla_{(I-J^2)Y}J)X+(\nabla_X J)Y+J(\nabla_X J^2)Y-J^2(\nabla_X J)Y+$\\

$-(I-J^2)\sharp_g ((\nabla_{JX}g)Y-(\nabla_X g)JY)-T^{\nabla}(JX,(I-J^2)Y)+JT^{\nabla}(X,(I-J^2)Y)=0$,\\

\noindent for all $X,Y \in C^{\infty}(TM)$, where we denoted $\flat_g$ by $g$ and the exterior differential associated to $\nabla$ acting on $g$ by $ (d^{\nabla}g)(X,Y):=({\nabla}_X g)(Y)-({\nabla}_Y g)(X)+g(T^{\nabla}(X,Y)).$
\end{proposition}

\begin{proposition} Let $(M,J,g)$ be a locally metallic Riemannian manifold. Then $\hat{J}_p$ is $\nabla$-integrable, for $\nabla$ the Levi-Civita connection of $g$.
\end{proposition}

\begin{proof}
From previous proposition, we have that the generalized product structure $\hat{J}_p$ is $\nabla$-integrable if and only if the following conditions are satisfied:\\

$ N_J=0$\\

$ (\nabla_{Y}J)X- (\nabla_{X}J)Y)=0$\\

$(\nabla_{X}J^*)J^*-(\nabla_{JX}J^*)=0$\\

$(\nabla_{(I-J^2)X}J^*)g(Y)- (\nabla_{(I-J^2)Y}J^*)g(X)=0$\\

$(\nabla_{(I-J^2)X}J^2)Y- (\nabla_{(I-J^2)Y}J^2)X=0$\\

$-(\nabla_{JX}J^2)Y-(\nabla_{(I-J^2)Y}J)X+(\nabla_X J)Y+J(\nabla_X J^2)Y-J^2(\nabla_X J)Y=0$\\

$-(\nabla_{JX}J^2)Y+(\nabla_{(I+J^2)Y}J)X-(\nabla_X J)Y+J(\nabla_X J^2)Y-J^2(\nabla_X J)Y=0$,\\

\noindent for all $X,Y \in C^{\infty}(TM).$ In particular, if $\nabla J=0 $, then $\hat{J}_p$ is $\nabla$-integrable.
\end{proof}

\begin{definition}
A generalized product structure $\hat{J}$ on $M$ is called \textit{anti-pseudo-calibrated} if it is $(\cdot,\cdot)$-anti-invariant and the bilinear symmetric form defined by $(\cdot,\hat{J}\cdot)$ on $TM$ is non-degenerate, where $$(X+\alpha,Y+\beta):=-\frac{1}{2}(\alpha(Y)-\beta(X))$$ is the natural symplectic structure on $TM\oplus T^*M$.
\end{definition}

\begin{remark}
The generalized product structure $\hat{J}_p$ is anti-pseudo-calibrated with respect to $(\cdot,\cdot)$.
\end{remark}

\begin{proposition} Let $\hat{J}_p$ be the generalized product structure defined by the metallic Riemannian structure $(J,g)$ on $M$. Then:
$$G(\sigma,\tau):=(\sigma,\hat{J}_p(\tau))$$
with $\sigma, \tau \in C^{\infty}(TM \oplus T^*M)$, is a neutral metric.

\end{proposition}

\begin{proof} Locally we can write $2G$ in block matrix form as:
$$\begin{pmatrix}
               g & -J \\
               -J & -(I-J^2)\sharp_g \\
         \end{pmatrix}$$
As $J$ is $g$-symmetric, pointwise, we can take $g=I$ and $J=\Lambda$ the diagonal matrix with eigenvalues ${\lambda}_1,...,{\lambda}_n$ which are solutions of the metallic equation ${\lambda}^2-p\lambda-q=0$. Then we get:
$$\begin{pmatrix}
               I & -\Lambda \\
               -\Lambda & p\Lambda+(q-1)I\\
         \end{pmatrix}.$$
In order to compute the indices of $2G$, we can use Gauss-Lagrange algorithm and by elementary operations on rows and columns of the matrix we get the form:
$$\begin{pmatrix}
               I & 0\\
              0 & -I+({\Lambda}^2-p\Lambda-qI)\\
         \end{pmatrix},$$
therefore:
$$\begin{pmatrix}
               I & 0\\
              0 & -I\\
         \end{pmatrix}$$
hence $2G$ has $n$ positive and $n$ negative eigenvalues and the proof is complete.
\end{proof}
\begin{proposition}
Let $(\hat{J}_p:=\begin{pmatrix}
               J & (I-J^2)\sharp_g \\
               \flat_g & -J^* \\
         \end{pmatrix}, \hat{g})$ be the generalized product structure induced by the metallic Riemannian structure $(J,g)$ on $M$ with $\hat{g}$ the Riemannian metric defined by (\ref{e}). Then:
         $$\hat{D}\hat{J}_p=0 \ \textit{if and only if} \ DJ=0 \ \textit{and} \ Dg=0.$$
\end{proposition}

\begin{proof}
Remark that $(\hat{D}_Y\hat{J}_p)X=(D_YJ)X+(D_Yg)X$, for any $X, Y \in C^{\infty} (TM)$ and $(\hat{D}_Y\hat{J}_p)\alpha=-p(D_Y(J {\sharp}_g))\alpha-(q-1)(D_Y {\sharp}_g )\alpha- (D_YJ^*)\alpha$, for any $Y \in C^{\infty}(TM)$ and $\alpha \in C^{\infty}(T^*M)$, therefore the statement. \end{proof}

\begin{remark} Starting with a metallic structure on a manifold, with minimal restrictions on $p$ and $q$, some other generalized metallic structures on its generalized tangent bundle can be constructed as follows.

The metallic structure $J$ on $M$ induces two almost product structures on $M$:
$$F^{\pm }:=\pm (\frac{2}{2\sigma _{p, q}-p}J-\frac{p}{2\sigma _{p, q}-p}I),$$
the almost product structures $F^{\pm }$ induce two generalized product structures on $TM\oplus T^*M$:
$$\hat F^{\pm }:=\begin{pmatrix}
               F^{\pm } & 0 \\
               0 & (F^{\pm })^* \\
         \end{pmatrix}$$
and the generalized product structures $\hat F^{\pm }$ induce two metallic structures on $TM\oplus T^*M$:
$$\hat J_{+,m}^{\pm}:=\pm \frac{2\sigma _{p, q}-p}{2}\hat F^{+ }+\frac{p}{2}I, \ \ \hat J_{-,m}^{\pm}:=\pm \frac{2\sigma _{p, q}-p}{2}\hat F^{- }+\frac{p}{2}I,$$
where $$\hat J_{+,m}^{+}=\hat J_{-,m}^{-}=\begin{pmatrix}
               J & 0 \\
               0 & J^* \\
         \end{pmatrix}$$
         and
$$\hat J_{+,m}^{-}=\hat J_{-,m}^{+}=\begin{pmatrix}
               -J+pI & 0 \\
               0 & -J^*+pI \\
         \end{pmatrix}.$$

The metallic structure $J$ on $M$ induces a generalized product structure on $TM\oplus T^*M$:
$$\hat{J}_p:=\begin{pmatrix}
               J & (I-J^2)\sharp_g \\
               \flat_g & -J^* \\
         \end{pmatrix}$$
and the generalized product structure $\hat{J}_p$ induces two generalized metallic structures on $TM\oplus T^*M$:
$$\hat J_{m}^{\pm}:=\pm \frac{2\sigma _{p, q}-p}{2}\hat{J}_p+\frac{p}{2}I,$$
namely,
$$\hat J_{m}^{+}=\begin{pmatrix}
               \frac{2\sigma _{p, q}-p}{2}J+\frac{p}{2}I & -(pJ+(q-1)I)\sharp_g \\
               \flat_g & -\frac{2\sigma _{p, q}-p}{2}J^*+\frac{p}{2}I \\
         \end{pmatrix}$$
         and
$$\hat J_{m}^{-}=\begin{pmatrix}
               -\frac{2\sigma _{p, q}-p}{2}J+\frac{p}{2}I & -(pJ+(q-1)I)\sharp_g \\
               \flat_g & \frac{2\sigma _{p, q}-p}{2}J^*+\frac{p}{2}I \\
         \end{pmatrix}.$$
\end{remark}

\subsection{Generalized complex structure induced by $(J,g)$}

Let $(J,g)$ be a metallic Riemannian structure on $M$ such that $J^2=pJ+qI$, $p,q\in \mathbb{R}$. Then $\hat{J}_c:=\begin{pmatrix}
               J & -(I+J^2)\sharp_g \\
               \flat_g & -J^* \\
         \end{pmatrix}$ is a generalized complex structure on $M$, that is $\hat{J}_c^2=-I$ \cite{na}.\\

A direct computation gives the following.

\begin{proposition}
The generalized complex structure $\hat{J}_c$ induced by the metallic Riemannian structure $(J,g)$ on $M$ is $\nabla$-integrable if and only if the following conditions are satisfied:\\

$ N_J+(I+J^2)\sharp_g(d^{\nabla}g)=0$\\

$(\nabla_{JX}g)Y- (\nabla_{JY}g)X+J^*( (\nabla_{X}g)Y- (\nabla_{Y}g)X)+g( (\nabla_{Y}J)X- (\nabla_{X}J)Y)+$\\

$+g(T^{\nabla}(X,JY)+T^{\nabla}(JX,Y))=0$\\

$(d^{\nabla}g)((I+J^2)Y,X)+(\nabla_{X}J^*)g(JY)-(\nabla_{JX}J^*)g(Y)=0$\\

$(\nabla_{(I+J^2)X}J^*)g(Y)- (\nabla_{(I+J^2)Y}J^*)g(X)=0$\\

$(\nabla_{(I+J^2)X}J^2)Y- (\nabla_{(I+J^2)Y}J^2)X-T^{\nabla}((I+J^2)X,(I+J^2)Y)+$\\

$-(I+J^2)\sharp_g((\nabla_{(I+J^2)X}g)Y- (\nabla_{(I+J^2)Y}g)X)=0$\\

$-(\nabla_{JX}J^2)Y+(\nabla_{(I+J^2)Y}J)X-(\nabla_X J)Y+J(\nabla_X J^2)Y-J^2(\nabla_X J)Y+$\\

$+(I+J^2)\sharp_g ((\nabla_{JX}g)Y-(\nabla_X g)JY)+T^{\nabla}(JX,(I+J^2)Y)-JT^{\nabla}(X,(I+J^2)Y)=0$,\\

\noindent for all $X,Y \in C^{\infty}(TM)$, where we denoted $\flat_g$ by $g$ and the exterior differential associated to $\nabla$ acting on $g$ by $ (d^{\nabla}g)(X,Y):=({\nabla}_X g)(Y)-({\nabla}_Y g)(X)+g(T^{\nabla}(X,Y)).$
\end{proposition}

\begin{proposition} Let $(M,J,g)$ be a locally metallic Riemannian manifold. Then $\hat{J}_c$ is $\nabla$-integrable, for $\nabla$ the Levi-Civita connection of $g$.
\end{proposition}

\begin{proof}
From the previous proposition, we have that the generalized complex structure $\hat{J}_c$ is $\nabla$-integrable if and only if the following conditions are satisfied:\\

$ N_J=0$\\

$ (\nabla_{Y}J)X- (\nabla_{X}J)Y=0$\\

$(\nabla_ X J^*)J^*-(\nabla_{JX}J^*)=0$\\

$(\nabla_{(I+J^2)X}J^*)g(Y)- (\nabla_{(I+J^2)Y}J^*)g(X)=0$\\

$(\nabla_{(I+J^2)X}J^2)Y- (\nabla_{(I+J^2)Y}J^2)X=0$\\

$-(\nabla_{JX}J^2)Y+(\nabla_{(I+J^2)Y}J)X-(\nabla_X J)Y+J(\nabla_X J^2)Y-J^2(\nabla_X J)Y=0$,\\

\noindent for all $X,Y \in C^{\infty}(TM).$ In particular, if $\nabla J=0 $, then $\hat{J}_c$ is $\nabla$-integrable.
\end{proof}

\begin{definition}
A generalized complex structure $\hat{J}$ on $M$ is called \textit{calibrated} if it is $(\cdot,\cdot)$-invariant and the bilinear symmetric form defined by $(\cdot,\hat{J}\cdot)$ on $TM$ is non-degenerate and positive definite, where $$(X+\alpha,Y+\beta):=-\frac{1}{2}(\alpha(Y)-\beta(X))$$ is the natural symplectic structure on $TM\oplus T^*M$.
\end{definition}

\begin{remark}
The generalized complex structure $\hat{J}_c$ is calibrated with respect to $(\cdot,\cdot)$.
\end{remark}

\begin{proposition}
Let $(\hat{J}_c:=\begin{pmatrix}
               J & -(I+J^2)\sharp_g \\
               \flat_g & -J^* \\
         \end{pmatrix}, \hat{g})$ be the generalized complex structure induced by the metallic Riemannian structure $(J,g)$ on $M$ with $\hat{g}$ the Riemannian metric defined by (\ref{e}). Then:
         $$\hat{D}\hat{J}_c=0 \ \textit{if and only if} \ DJ=0 \ \textit{and} \ Dg=0.$$

\end{proposition}

\begin{proof}
Remark that $(\hat{D}_Y\hat{J}_c)X=(D_YJ)X+(D_Yg)X$, for any $X, Y \in C^{\infty} (TM)$ and $(\hat{D}_Y\hat{J}_c)\alpha=-p(D_Y(J {\sharp}_g))\alpha-(q+1)(D_Y {\sharp}_g )\alpha- (D_YJ^*)\alpha$, for any $Y \in C^{\infty}(TM)$ and $\alpha \in C^{\infty}(T^*M)$, therefore the statement.
\end{proof}

\begin{definition}
A pair $(\hat{J}_c,\hat{J}_p)$ of a generalized complex structure and a generalized product structure is called \textit{generalized complex product structure} if $\hat{J}_c\hat{J}_p=-\hat{J}_p\hat{J}_c$.
\end{definition}

\begin{remark}
If $(J,g)$ is a metallic Riemannian structure on $M$, then $(\hat{J}_c,\hat{J}_p)$, for $\hat{J}_c:=\begin{pmatrix}
               J & -(I+J^2)\sharp_g \\
               \flat_g & -J^* \\
         \end{pmatrix}$ and $\hat{J}_p:=\begin{pmatrix}
               J & (I-J^2)\sharp_g \\
               \flat_g & -J^* \\
         \end{pmatrix}$, is a generalized complex product structure.
\end{remark}

\section{Metallic structures on tangent and cotangent bundles}

\subsection{Metallic structure on the tangent bundle}

Let $(M,J,g)$ be a metallic Riemannian manifold and let $\nabla$ be a linear connection on $M$. $\nabla$ defines the decomposition into the horizontal and vertical subbundles of $T(TM)$:
$$ T(TM)=T^H(TM)\oplus T^V(TM).$$

Let $\pi :TM \rightarrow M$ be the canonical projection and ${\pi}_*:T(TM) \rightarrow TM$ be the tangent map of $\pi$. If $a \in TM$ and $A \in T_a(TM)$, then ${\pi}_*(A) \in T_{\pi (a)}M$ and we denote by ${\chi}_a$ the standard identification between $T_{\pi (a)}M$ and its tangent space $ T_a(T_{\pi (a)}M)$.

Let ${\Psi}^{\nabla} :TM \oplus T^{*}M \rightarrow T(TM)$ be the bundle morphism defined by:
$${\Psi}^{\nabla}(X+\alpha):=X^H_a+{\chi}_a({\sharp}_g \alpha),$$
where $a \in TM$ and $X^H_a$ is the horizontal lifting of $X\in T_{\pi (a)}M$.

Let $\left\{ x^{1},...,x^{n}\right\} $ be local coordinates on $M$, let $\left\{ {\tilde{x}}^{1},..., {\tilde{x}}^{n},y^1,...,y^n\right\} $ be respectively the corresponding local coordinates on $TM$ and let $\{X_1,...,X_n, \dfrac{\partial }{\partial
{y^{1}}},.., \dfrac{\partial }{\partial
{y^{n}}}\}$ be a local frame on $T(TM)$, where $X_i=\dfrac{\partial }{\partial
{\tilde{x}}^{i}}$. We have:
$$ X_i^H=X_i-y^k{\Gamma}^l_{ik}{\dfrac{\partial }{\partial
y^{l}}}$$
$$ X_i^V=y^k{\Gamma}^l_{ik}{\dfrac{\partial }{\partial
y^{l}}}$$
$$\left( \dfrac{\partial }{\partial
{y^{i}}}\right)^H=0$$
$$\left ( \dfrac{\partial }{\partial
{y^{i}}}\right)^V= \dfrac{\partial }{\partial
{y^{i}}}
$$
where $i,k,l$ run from $1$ to $n$ and $\Gamma^k_{il}$ are the Christoffel's symbols of $\nabla$.

Let ${\Psi}^{\nabla} :TM \oplus T^{*}M \rightarrow T(TM)$ be the bundle morphism defined before (which is an isomorphism on the fibres). In local coordinates, we have the following expressions:
$${\Psi}^{\nabla}\left(\dfrac{\partial }{\partial
{x^{i}}}\right)=X_i^H$$
$${\Psi}^{\nabla}\left({dx^j}\right)=g^{jk} \dfrac{\partial}{\partial y^k}.$$

Let $({\hat J}_m,\hat g)$ be the generalized metallic structure defined in the previous section. The isomorphism ${\Psi}^{\nabla}$ allows us to construct a natural metallic structure $\bar J_m$ and a natural Riemannian metric $\bar g$ on $TM$ in the following way.

We define $\bar J_m:T(TM) \rightarrow T(TM)$ by
$${\bar J}_m:= ({\Psi}^{\nabla} )\circ{\hat J}_m \circ ({\Psi}^{\nabla} )^{-1}$$
and the Riemannian metric $\bar g$ on $TM$ by
$${\bar g}:= (({\Psi}^{\nabla} )^{-1})^{*}(\hat g). $$

\begin{proposition} $(TM,{\bar J}_m,\bar g)$ is a metallic Riemannian manifold.
\end{proposition}

\begin{proof}
From the definition it follows that ${\bar J}_m^2=p{\bar J}_m+qI$ and ${\bar g}({\bar J}_mX,Y)={\bar g}(X,{\bar J}_mY)$, for any $X,Y \in C^{\infty}(T(TM))$.
\end{proof}

In local coordinates, we have the following expressions for ${\bar J}_m$ and $\bar g$:
$$
\left \{
\begin{array}{l}
\vspace{0.2cm}
{\bar J}_m \left(  X_i^H\right)=J^k_iX_k^H\\
{\bar J}_m \left({\dfrac{\partial }{\partial
y^j}}\right)=g_{ji}J_k^i g^{kl}{\dfrac{\partial }{\partial
y^{l}}}=J^k_j{\dfrac{\partial }{\partial
y^{k}}}
\end{array}\label{190}
\right.
$$
$$
\left \{
\begin{array}{l}
\vspace{0.2cm}
{\bar g} \left(  X_i^H,X^H_j\right)=g_{ij}\\
\vspace{0.2cm}
{\bar g} \left(  X_i^H,{\dfrac{\partial }{\partial
y^{j}}}\right )=0\\
{\bar g} \left({\dfrac{\partial }{\partial
y^{i}}},{\dfrac{\partial }{\partial
y^{j}}}\right)=g_{ij}.
\end{array}\label{200}
\right.
$$

Moreover:
$$
{\bar J}_m \left(  X_i\right)=J^k_i X_k-y^l(J^k_i {\Gamma}^s_{kl}- J^s_r{\Gamma}^r_{il} ){\dfrac{\partial }{\partial
y^{s}}}\label{190}
$$
$$
\left \{
\begin{array}{l}
\vspace{0.2cm}
{\bar g} \left(  X_i,X_j\right)=g_{ij}+y^ky^h{\Gamma}^l_{ik}{\Gamma}^s_{jh}g_{hk}\\
\vspace{0.2cm}
{\bar g} \left(  X_i,{\dfrac{\partial }{\partial
y^{j}}}\right )=y^k{\Gamma}^l_{ik}g_{lj}.\\
\end{array}\label{200}
\right.
$$

Computing the Nijenhuis tensor of ${\bar J}_m$, we get the following:
$$
\begin{array}{l}
\vspace{0.2cm}
N_{{\bar J}_m}\left({\dfrac{\partial }{\partial y^{i}}},{\dfrac{\partial }{\partial y^{j}}}\right)=0 \\
\end{array}\label{210}
$$
$$
\begin{array}{l}
\vspace{0.2cm}
N_{{\bar J}_m} \left( X^H_{i},{\dfrac{\partial }{\partial y^{j}}}\right)={\left( \left( {\nabla}_{JX_{i}} J \right) X_{j}-J \left( {\nabla}_{X_{i}} J \right){ X_{j}}\right)}^k {\dfrac{\partial }{\partial y^{k}}}\\
\end{array} \label{220}
$$
$$
\begin{array}{l}
\vspace{0.2cm}
N_{{\bar J}_m}\left(  X_i^H, X_j^H \right)=(N_J\left( {X_{i}},X_j\right))^kX_k^H +\\
- y^s \left(J^k_iJ^h_jR^r_{khs} -J_l^rJ^k_iR^l_{kjs}-J^h_jJ^r_lR^l_{ihs}+pJ^r_lR^l_{ijs}+qR^r_{ijs}\right) {{\dfrac{\partial }{\partial y^{r}}}}. \\
\end{array}\label{230}
$$

Therefore we can state the following.

\begin{proposition} Let $(M,J,g)$ be a flat locally metallic Riemannian manifold. If $\nabla$ is the Levi-Civita connection of $g$, then $({\bar J}_m,\bar g)$ is an integrable metallic Riemannian structure on $TM$.
\end{proposition}

\subsection{Metallic structure on the cotangent bundle}

Let $(M,J,g)$ be a metallic Riemannian manifold and let $\nabla$ be a linear connection on $M$. $\nabla$ defines the decomposition into the horizontal and vertical subbundles of $T(T^*M)$:
$$ T(T^*M)=T^H(T^*M)\oplus T^V(T^*M).$$

Let $\pi :T^*M \rightarrow M$ be the canonical projection and ${\pi}_*:T(T^*M) \rightarrow TM$ be the tangent map of $\pi$. If $a \in T^*M$ and $A \in T_a(T^*M)$, then ${\pi}_*(A) \in T_{\pi (a)}M$ and we denote by ${\chi}_a$ the standard identification between $T^*_{\pi (a)}M$ and its tangent space $ T_a(T^*_{\pi (a)}M)$.

Let ${\Phi}^{\nabla} :TM \oplus T^{*}M \rightarrow T(T^*M)$ be the bundle morphism defined by \cite{an}:
$${\Phi}^{\nabla}(X+\alpha):=X^H_a+{\chi}_a(\alpha),$$
where $a \in T^*M$ and $X^H_a$ is the horizontal lifting of $X\in T_{\pi (a)}M$.

Let $\left\{ x^{1},...,x^{n}\right\} $ be local coordinates on $M$, let $\left\{ {\tilde{x}}^{1},..., {\tilde{x}}^{n},y_1,...,y_n\right\} $ be respectively the corresponding local coordinates on $T^*M$ and let $\{X_1,...,X_n, \dfrac{\partial }{\partial
{y_{1}}},.., \dfrac{\partial }{\partial
{y_{n}}}\}$ be a local frame on $T(T^*M)$, where $X_i=\dfrac{\partial }{\partial
{\tilde{x}}^{i}}$. We have:
$$ X_i^H=X_i+y_k{\Gamma}^k_{il}{\dfrac{\partial }{\partial
y_{l}}}$$
$$ X_i^V=-y_k{\Gamma}^k_{il}{\dfrac{\partial }{\partial
y_{l}}}$$
$$\left( \dfrac{\partial }{\partial
{y_{i}}}\right)^H=0$$
$$\left ( \dfrac{\partial }{\partial
{y_{i}}}\right)^V= \dfrac{\partial }{\partial
{y_{i}}}
$$
where $i,k,l$ run from $1$ to $n$ and $\Gamma^k_{il}$ are the Christoffel's symbols of $\nabla$.

Let ${\Phi}^{\nabla} :TM \oplus T^{*}M \rightarrow T(T^*M)$ be the bundle morphism defined before (which is an isomorphism on the fibres). In local coordinates, we have the following expressions:
$${\Phi}^{\nabla}\left(\dfrac{\partial }{\partial
{x^{i}}}\right)=X_i^H$$
$${\Phi}^{\nabla}\left({dx^j}\right)=\dfrac{\partial }{\partial
{y_{j}}}.$$

Let $({\hat J}_m,\hat g)$ be the generalized metallic structure defined in the previous section. The isomorphism ${\Phi}^{\nabla}$ allows us to construct a natural metallic structure $\tilde J_m$ and a natural Riemannian metric $\tilde g$ on $T^*M$ in the following way.

We define $\tilde J_m:T(T^*M) \rightarrow T(T^*M)$ by
$${\tilde J}_m:= ({\Phi}^{\nabla} )\circ{\hat J}_m \circ ({\Phi}^{\nabla} )^{-1}$$
and the Riemannian metric $\tilde g$ on $T^*M$ by
$${\tilde g}:= (({\Phi}^{\nabla} )^{-1})^{*}(\hat g). $$

\begin{proposition}
$(T^*M,{\tilde J}_m,\tilde g)$ is a metallic Riemannian manifold.
\end{proposition}

\begin{proof}
From the definition it follows that ${\tilde J}_m^2=p{\tilde J}_m+qI$ and ${\tilde g}({\tilde J}_mX,Y)={\tilde g}(X,{\tilde J}_mY)$, for any $X,Y \in C^{\infty}(T(T^*M))$.
\end{proof}

In local coordinates, we have the following expressions for ${\tilde J}_m$ and $\tilde g$:
$$
\left \{
\begin{array}{l}
\vspace{0.2cm}
{\tilde J}_m \left(  X_i^H\right)=J^k_iX_k^H\\
{\tilde J}_m \left({\dfrac{\partial }{\partial
y_{j}}}\right)=J^j_k{\dfrac{\partial }{\partial
y_{k}}}
\end{array}\label{190}
\right.
$$
$$
\left \{
\begin{array}{l}
\vspace{0.2cm}
{\tilde g} \left(  X_i^H,X^H_j\right)=g_{ij}\\
\vspace{0.2cm}
{\tilde g} \left(  X_i^H,{\dfrac{\partial }{\partial
y_{j}}}\right )=0\\
{\tilde g} \left({\dfrac{\partial }{\partial
y_{i}}},{\dfrac{\partial }{\partial
y_{j}}}\right)=g^{ij}.
\end{array}\label{200}
\right.
$$

Moreover:
$$
{\tilde J}_m \left(  X_i\right)=J^k_i X_k+y_l(J^k_i {\Gamma}^l_{kr}- J^s_r{\Gamma}^l_{is} ){\dfrac{\partial }{\partial
y_{r}}}\label{190}
$$
$$
\left \{
\begin{array}{l}
\vspace{0.2cm}
{\tilde g} \left(  X_i,X_j\right)=g_{ij}+y_ky_h{\Gamma}^k_{il}{\Gamma}^h_{jr}g^{lr}\\
\vspace{0.2cm}
{\tilde g} \left(  X_i,{\dfrac{\partial }{\partial
y_{j}}}\right )=-y_k{\Gamma}^k_{il}g^{lj}.\\
\end{array}\label{200}
\right.
$$

Computing the Nijenhuis tensor of ${\tilde J}_m$, we get the following:
$$
\begin{array}{l}
\vspace{0.2cm}
N_{{\tilde J}_m}\left({\dfrac{\partial }{\partial y_{i}}},{\dfrac{\partial }{\partial y_{j}}}\right)=0 \\
\end{array}\label{210}
$$
$$
\begin{array}{l}
\vspace{0.2cm}
N_{{\tilde J}_m} \left( X^H_{i},{\dfrac{\partial }{\partial y_{j}}}\right)={\left( \left( {\nabla}_{JX_{i}} J \right) X_{k}-J\left( {\nabla}_{X_{i}} J \right){ X_{k}}\right)}^j {\dfrac{\partial }{\partial y_{k}}}\\
\end{array} \label{220}
$$
$$
\begin{array}{l}
\vspace{0.2cm}
N_{{\tilde J}_m}\left(  X_i^H, X_j^H \right)=(N_J\left( {X_{i}},X_j\right))^kX_k^H +\\
+ y_l \left(J^k_iJ^h_jR^l_{khs} -J_s^rJ^k_iR^l_{kjr}-J^r_sJ^k_jR^l_{ikr}+pJ^k_sR^l_{ijk}+qR^l_{ijs}\right) {{\dfrac{\partial }{\partial y_{s}}}}. \\
\end{array}\label{230}
$$

Therefore we can state the following.

\begin{proposition} Let $(M,J,g)$ be a flat locally metallic Riemannian manifold. If $\nabla$ is the Levi-Civita connection of $g$, then $({\tilde J}_m,\tilde g)$ is an integrable metallic Riemannian structure on $T^*M$.
\end{proposition}

\begin{remark}
The metallic structures ${\bar J}_m$ and ${\tilde J}_m$ on the tangent and cotangent bundles respectively, satisfy:
$${\bar J}_m \circ ({\Psi}^{\nabla}\circ ({\Phi}^{\nabla})^{-1})=({\Psi}^{\nabla}\circ ({\Phi}^{\nabla})^{-1})\circ {\tilde J}_m.$$
\end{remark}

\bigskip

\textit{Adara M. Blaga}

\textit{Department of Mathematics}

\textit{West University of Timi\c{s}oara}

\textit{Bld. V. P\^{a}rvan nr. 4, 300223, Timi\c{s}oara, Rom\^{a}nia}

\textit{adarablaga@yahoo.com}

\bigskip

\textit{Antonella Nannicini}

\textit{Department of Mathematics and Informatics "U. Dini"}

\textit{University of Florence}

\textit{Viale Morgagni, 67/a, 50134, Firenze, Italy}

\textit{antonella.nannicini@unifi.it}

\bigskip

\end{document}